\documentclass[a4paper, 10pt]{amsart}
\usepackage{amsmath,amssymb,amsfonts,amsthm,latexsym,mathrsfs}	
\usepackage[utf8]{inputenc}
\usepackage{lmodern}
\usepackage[T1]{fontenc}
\usepackage{microtype} 	
\usepackage{enumerate}
\usepackage{url}
\usepackage{arydshln,color}
\usepackage{booktabs}
\theoremstyle{plain}
\newtheorem{theorem}[subsection]{{\bf Theorem}}
\newtheorem*{theorem*}{{\bf Theorem}}
\newtheorem{corollary}[subsection]{{\bf Corollary}}
\newtheorem*{corollary*}{{\bf Corollary}}
\newtheorem{proposition}[subsection]{{\bf Proposition}}
\newtheorem{lemma}[subsection]{{\bf Lemma}}

\theoremstyle{definition}

\theoremstyle{remark}

\numberwithin{equation}{subsubsection}
\def\ZZ{\mathbb Z}
\def\CC{\mathbb C}
\def\FF{\mathbb F}

\DeclareMathOperator\UT{UT}
\DeclareMathOperator\GL{GL}
\DeclareMathOperator\HH{H}
\DeclareMathOperator\MM{M}

\begin{document}
\baselineskip=15pt
\title{Schur multipliers of unitriangular groups}
\author{Urban Jezernik}
\address{Institute of Mathematics, Physics, and Mechanics \\ Ljubljana \\ Slovenia}
\email{urban.jezernik@imfm.si}
\date{\today}
\begin{abstract}
An explicit formula for the Schur multiplier of the group of unitriangular matrices over products of cyclic rings $\ZZ/m\ZZ$ and $\ZZ$ is derived. We use it to provide presentations of the corresponding covering groups and touch upon the question of finding their nonisomorphic representatives.
\end{abstract}
\maketitle
\section{Introduction}
\noindent 
Let $G$ be a group. The second integral homology group $\HH_2(G, \ZZ)$, also denoted by $\MM(G)$, is known as the \emph{Schur multiplier} of $G$. When the group $G$ is finite, $\MM(G)$ is isomorphic to the cohomology group $\HH^2(G, \CC)$. This group has several important applications in the theory of projective representations and the theory of central extensions \cite{Kar87}, and its implicit study goes back to Schur \cite{Sch07}.

The Schur multiplier of the general linear group $\GL_n(\FF_p)$ over the finite field with $p$ elements has been computed only recently while determining its non-abelian tensor square \cite{Erfanian08}. It turns out to be trivial whenever $(n,p) \notin \{ (3,2), (4,2) \}$, and isomorphic to $\ZZ/2\ZZ$ otherwise. These special cases can easily be dealt with using techniques of \cite{Nickel93} with the help of computational tools \cite{GAP}.

On the other hand, the structure of the multipliers of the Sylow subgroups of $\GL_n( \FF_p)$ is a lot more abundant. The Sylow $p$-subgroup corresponds to the unitriangular group $\UT_n( \FF_p)$, and the rest are given as wreath products of the unitriangular group and cyclic groups \cite{AlAli05}. Using Blackburn's results \cite{Blackburn72}, calculating the multipliers of the latter groups is routine once the multiplier of the unitriangular group is determined. This has been done before for the mod-$p$-multiplier of unitriangular groups over rings $\ZZ/m\ZZ$ for odd $m$ using a more homological approach, which seems to become overly complicated in the general case, c.f. \cite{Evens72, Biss01}. 

We complete this task by providing an explicit formula for the Schur multiplier of the unitriangular group over $\ZZ/m\ZZ$ for all integers $m$. The methods used extend to the case when the ring is a direct product of cyclic rings $\ZZ/m\ZZ$ and $\ZZ$. The result goes as follows.

\begin{theorem} \label{t:main-Intro}
The Schur multiplier of $\UT_n(\ZZ/m\ZZ)$ is isomorphic to
	\[
		\textstyle (\ZZ/m\ZZ)^{\binom{n}{2} - 1} \text{ for odd $m$,} \qquad (\ZZ/2\ZZ)^{n-3} \oplus (\ZZ/\frac{m}{2}\ZZ)^{n-2} \oplus (\ZZ/m\ZZ)^{\binom{n-1}{2}}  \text{ for even $m$.}
	\]
The Schur multiplier of $\UT_n(\prod_{i=1}^r \ZZ/m_i\ZZ)$, where $m_i | m_{i+1}$, is isomorphic to
 	\[
		\bigoplus_{i=1}^r \MM(\UT_n( \ZZ/m_i\ZZ)) \oplus \bigoplus_{i=1}^{r-1} (\ZZ/{m_i}\ZZ)^{i(n-1)^2}.
	\]
\end{theorem}

In general, there are two basic approaches in determining the Schur multiplier of a group -- either using a combinatorial description of the multiplier, provided by Hopf's integral homology formula \cite{Kar87}, or using algorithms from the theory of polycyclic groups \cite{Nickel93}. The latter are more suited for computer computations, while Hopf's formula requires a neat presentation of the given group. 
We rely here on some fairly recent results of \cite{Biss01}, where certain presentations of the unitriangular groups in terms of generators and relators are given. These presentations are proved to be minimal for the groups $\UT_n(\ZZ/m\ZZ)$ with $m$ odd. We use this presentation and Hopf's formula to obtain a description of $\MM(\UT_n(\ZZ/m\ZZ))$ first, and then generalize the technique of the calculation to the case of unitriangular groups over products of such rings. As a consequence, Theorem \ref{t:main-Intro} may be applied to the work of \cite{Biss01} in proving that the presentations of the groups $\UT_n(\ZZ/2\ZZ)$ and $\UT_n(\ZZ)$ are also minimal.

Another advantage of the above approach is that it provides a simple way to obtain free presentations of all covering groups of the unitriangular groups $\UT_n(\FF_p)$. This in a way solves a problem posed by Berkovich in \cite{Ber08}, asking for the multiplier of $\UT_n(\FF_p)$ and a description of its covering groups. The remaining issue is determining isomorphism representatives of these groups. It is well-known that they all belong to the same isoclinism family, and we show that in the case $n=3$, any stem group of the family is a covering group and there are precisely $p+7$ of them whenever $p > 3$. It is not so for $n > 3$, where the appearance of an additional central element causes certain complications. We provide experimental data, which indicates that the situation gets somewhat out of control in this case. 

Our methods do not seem to be applicable in a more general setting of determining the multipliers of unitriangular groups over any finite field. In principle, this could be done by tackling the problem with the theory of central extensions of polycyclic groups \cite{Nickel93}, since unitriangular groups admit an efficient polycyclic presentation. Nonetheless, these techniques quickly become far too complicated to do them by hand. 
\section{Multipliers over cyclic rings}
\noindent We first deal with unitriangular groups over the rings $\ZZ/m\ZZ$ for any integer $m$. Recall from \cite{Biss01} that a presentation of $\UT_n(\ZZ/m\ZZ)$ is given by the set of generators $\mathcal S = \{ s_i \mid 1 \leq i \leq n-1 \}$ subject to the set $\mathcal R$ of relators in Table \ref{tab:relators}. Furthermore, relators in the final row of this table are unnecessary for odd $m$. Here, the generators $s_i$ correspond to elementary matrices $I + E_{i,i+1}$ in the unitriangular group.
 
\begin{table}[h]
	\caption{Relators in the presentation of $\UT_n(\ZZ/m\ZZ)$.}
	\label{tab:relators}
	\begin{center}
		\begin{tabular}{lr}
\toprule
 $s_i^m$ 			
 & $1 \leq i \leq n-1$ \\
 
 $[s_i, s_j]$  	
 & $1 \leq i < j-1 \leq n-2$ \\

 $[s_i, s_{i+1}, s_i] $
 & $1 \leq i \leq n-2$ \\

 $[s_i, s_{i+1}, s_{i+1}] $
 & $1 \leq i \leq n-2$ \\

 $[[s_i, s_{i+1}], [s_{i+1}, s_{i+2}]]$ 
 & $1 \leq i \leq n-3$  \\ \bottomrule
		\end{tabular}
	\end{center}
\end{table}

Let $F$ be the free group on $\mathcal S$ and $R$ the normal subgroup of $F$ generated by $\mathcal R$, so that $F/R$ is a free presentation of the group $\UT_n(\ZZ/m\ZZ)$. Its Schur multiplier is then given by the formula $(R \cap [F,F])/[R,F]$. We first restrict the orders of some special elements of this group, gathered in Table \ref{tab:1}.

\begin{table}[htb]
\caption{Generators of $\MM(\UT_n(\ZZ/m\ZZ))$ and their orders.}
\label{tab:1}
	\begin{center}
		\begin{tabular}{clr}
\toprule
$(1)$	& $[s_i, s_j]$
	& $m$ \\ 

$(2)$	& $[s_i, s_{i+1}, s_i]$
	& $m$ \\

	& $[s_i, s_{i+1}, s_i s_{i+1}^{-1}]$
	& $\gcd(m, \binom{m}{2})$ \\ 

$(3)$	& $[[s_i, s_{i+1}], [s_{i+1}, s_{i+2}]]$
	& $\gcd(2,m)$ \\ \bottomrule
		\end{tabular}
	\end{center}
\end{table}

\begin{lemma} \label{lem:orders}
 In the group $F/[R,F]$, elements of Table \ref{tab:1} have orders dividing the numbers in the far right column.
\end{lemma}
\begin{proof}
 All calculations are done in the group $F/[R,F]$. As the given relators are central, we have $[s_i, s_j]^m = [s_i^m, s_j] = 1$ and likewise $[s_i, s_{i+1}, s_i]^m = 1$. Note that $[s_i, s_{i+1}]^m [s_i, s_{i+1}, s_i]^{\binom{m}{2}} = [s_i, s_{i+1}^m] = 1$ and $[s_i, s_{i+1}]^m [s_i, s_{i+1}, s_{i+1}]^{\binom{m}{2}} = 1$. The commutator $[s_i, s_{i+1}, s_i s_{i+1}^{-1}]$ is thus of order dividing $\gcd(m, \binom{m}{2})$, which equals $m$ for odd $m$ and $m/2$ for even $m$. This takes care of $(1)$ and $(2)$. For $(3)$, use the Hall-Witt identity \cite{Kar87} on the elements $[s_{i+2}, s_{i+1}]$, $s_{i+1}$ and $s_i$ to obtain
\[
[s_{i+2}, s_{i+1}, s_{i+1}, s_i]^{s_{i+1}^{-1}} \cdot [s_{i+1}^{-1}, s_i^{-1}, [s_{i+2}, s_{i+1}]]^{s_i} \cdot [s_i, [s_{i+1}, s_{i+2}], s_{i+1}^{-1}]^{[s_{i+3}, s_{i+2}]} = 1.
\]
The first commutator is trivial, whereas the second and third one can be simplified as
\[
 [s_{i+1}^{-1}, s_i^{-1}, [s_{i+2}, s_{i+1}]]^{s_i}  = [[s_{i+1}, s_i]^{[s_i, s_{i+1}] \cdot s_{i+1}^{-1}}, [s_{i+2}, s_{i+1}]^{s_i}] 
= [[s_{i+1}, s_i], [s_{i+2}, s_{i+1}]]
\]
and
\begin{align*}
& [s_i, [s_{i+1}, s_{i+2}], s_{i+1}^{-1}]^{[s_{i+2}, s_{i+1}]} \\
&\quad = [[s_i, [s_{i+2}, s_{i+1}]]^{- [s_{i+1}, s_{i+2}]}, s_{i+1}^{-1}]^{[s_{i+2}, s_{i+1}]} \\
&\quad = [s_{i+2}, s_{i+1}, s_i, s_{i+1}^{-[s_{i+2}, s_{i+1}]}] \\
&\quad = [s_{i+2}, s_{i+1}, s_i, s_{i+1}]^{-1},
\end{align*}
hence $[[s_i, s_{i+1}], [s_{i+1}, s_{i+2}]] \cdot [s_{i+2}, s_{i+1}, s_i, s_{i+1}]^{-1} = 1$. On the other hand, the relator in question may be expressed as
\begin{align*}
& [[s_i, s_{i+1}], [s_{i+1}, s_{i+2}]] \\
&\quad = [s_i, s_{i+1}, s_{i+1}^{-1} s_{i+1}^{s_{i+2}}]  \\
&\quad = [s_i, s_{i+1}, s_{i+1}^{s_{i+2}}] \cdot [s_i, s_{i+1}, s_{i+1}^{-1}]^{s_{i+1} [s_{i+1}, s_{i+2}]}  \\
&\quad = [[s_i, s_{i+1}]^{s_{i+2}^{-1}}, s_{i+1}]^{s_{i+2}} \cdot  [s_i, s_{i+1}, s_{i+1}]^{-1} \\
&\quad = [[s_i, s_{i+1}] \cdot [s_i, s_{i+1}, s_{i+2}]^{- s_{i+2}^{-1}},  s_{i+1}]^{s_{i+2}} \cdot [s_i, s_{i+1}, s_{i+1}]^{-1} \\
&\quad = [[s_i, s_{i+1}, s_{i+2}]^{- s_{i+2}^{-1}}, s_{i+1}]^{s_{i+2}} \\
&\quad = [s_i, s_{i+1}, s_{i+2}, s_{i+1}]^{- [s_{i+2}, [s_i, s_{i+1}]] s_{i+2}} \\
&\quad = [s_i, s_{i+1}, s_{i+2}, s_{i+1}]^{-1}.
\end{align*}
Combining the two, we get
\[
  [[s_i, s_{i+1}], [s_{i+1}, s_{i+2}]]^2 = [[s_i, s_{i+1}, s_{i+2}]^{-1} \cdot [s_{i+2}, s_{i+1}, s_i], s_{i+1}] = 1,
\]
where the last equality comes from the fact that the left commutant of the last commutator is an element of $R$. This can be checked in the group $F/R = \UT_n(\ZZ/m\ZZ)$ itself by identifying $s_i$ with $I + E_{i,i+1}$. The proof is complete.
\end{proof}

Next, we show that these elements form a generating set of the multiplier.

\begin{lemma} \label{lem:expansion}
 Relators of Table \ref{tab:1} generate the group $(R \cap [F,F])/[R,F]$.
\end{lemma}
\begin{proof}
 Let $w \in R \cap [F,F]$. Expand it as a product of conjugates of elements of $\mathcal R$. Using the previous lemma, we may assume $w$ can be written as a product of relators of Table \ref{tab:1} multiplied by $\prod_i s_i^{m e_i}$ for some integers $e_i$. As $w$ is also contained in $[F,F]$, all $e_i$ must indeed be zero. 
\end{proof}

We now prove that these relators are also independent modulo $[R,F]$.

\begin{lemma} \label{lem:unique}
 Whenever a product of relators of Table \ref{tab:1} is an element of $[R, F]$, the corresponding exponents are divisible by the numbers in the far right column of Table \ref{tab:1}.
\end{lemma}
\begin{proof}
Let $w$ be the product
\[
\prod_{i < j-1} [s_i, s_j]^{a_{ij}} \cdot \prod_{i} [s_i, s_{i+1}, s_i]^{b_i} \cdot \prod_{i} [s_i, s_{i+1}, s_i s_{i+1}^{-1}]^{c_i} \cdot \prod_{i} [[s_i, s_{i+1}], [s_{i+1}, s_{i+2}]]^{d_i}
\]
for some nonnegative integers $a_{ij}, b_i < m$, $c_i < \gcd(\binom{m}{2}, m)$ and $d_i < \gcd(2,m)$, and assume that it is an element of $[R,F]$. Modulo $\gamma_3(F)$, it equals $\prod_{i < j-1} [s_i, s_j]^{a_{ij}}$. As $w$ is contained in $[R,F]$, it can also be written as a product of conjugates of commutators, the first commutant of which is a relator. This amounts to $\prod_{i,j} [s_i^m, s_j]^{\pm 1} \equiv \prod_{i<j} [s_i, s_j]^{\pm m}$ modulo $\gamma_3(F)$. Basic commutators of weight $2$ form a basis of the free abelian group $\gamma_2(F)/\gamma_3(F)$, so each $a_{ij}$ is divisible by $m$, hence zero by the restriction $a_{ij} < m$. Modulo $\gamma_4(F)$, the product $w$ now equals $\prod_{i} [s_i, s_{i+1}, s_i]^{b_i} \cdot \prod_{i} [s_i, s_{i+1}, s_i s_{i+1}^{-1}]^{c_i}$. As before, any element of $[R,F]$ can be expressed as
\[
 \prod_{i,j} [s_i^m, s_j]^{\pm *} \cdot \prod_{i < j-1, k} [s_i, s_j, s_k]^{\pm 1}
\]
modulo $\gamma_4(F)$, where the $*$ symbolizes any element of $F$. Everything except commutators of weight three in which two consecutive generators appear must cancel out, so $w$ equals
\[
 \prod_{i} [s_i, s_{i+1}]^{\pm m*} [s_i, s_{i+1}, s_i]^{\pm \binom{m}{2}} \cdot \prod_{i} [s_{i+1}, s_i]^{\pm m*} [s_i, s_{i+1}, s_{i+1}]^{\mp \binom{m}{2}}, 
\]
modulo $[R,F]$, the first part of which derives from commutators $[s_i^m, s_{i+1}] = [s_i, s_{i+1}]^m \cdot [s_i, s_{i+1}, s_i]^{\binom{m}{2}}$ and the second one analogously from $[s_{i+1}^m, s_i] = [s_{i+1}, s_i]^m \cdot [s_{i+1}, s_i, s_{i+1}]^{\binom{m}{2}} = [s_{i+1}, s_i]^m \cdot [s_i, s_{i+1}, s_{i+1}]^{-\binom{m}{2}}$. Canceling out the commutators of weight $2$, we are left with
\[
 \prod_i [s_i, s_{i+1}, s_i]^{m \alpha_i} [s_i, s_{i+1}, s_{i+1}]^{m \beta_i} [s_i, s_{i+1}, s_i s_{i+1}^{-1}]^{\binom{m}{2} \gamma_i}
\]
for some integers $\alpha_i, \beta_i, \gamma_i$. Comparing the basis expansion in $\gamma_3(F)/\gamma_4(F)$, we obtain $b_i + c_i = m \alpha_i - \binom{m}{2}\gamma_i$ and $c_i = -m \beta_i + \binom{m}{2}\gamma_i$. It follows that $\gcd(m, \binom{m}{2})$ divides the $c_i$, so the restriction on the $c_i$ implies they must all be zero. Summing the two equations, we conclude that $m$ divides the $b_i$, so they too must all be zero. What now remains of the product $w$ is only $\prod_i [[s_i, s_{i+1}], [s_{i+1}, s_{i+2}]]^{d_i}$. When $m$ is odd, the relator $[[s_i, s_{i+1}], [s_{i+1}, s_{i+2}]]$ is trivial in the group $F/[R,F]$. We are therefore left with the even case, and the proof of this in done by induction on $n$. It is easy to see that we have $\MM(\UT_4(\ZZ/2\ZZ)) \cong (\ZZ/2\ZZ)^4$, e.g. use \cite{Nickel93, GAP}. Now let $n > 4$ and assume, for the sake of contradiction, that not all of the $d_i$ are zero. If $d_1 = 0$, then by inductive hypothesis, all the remaining $d_i$'s are also zero; therefore $d_1  = 1$. The canonical epimorphism $\UT_n(\ZZ/m\ZZ) \to \UT_{n-1} (\ZZ/m\ZZ)$ may be composed into $\UT_n( \ZZ/m\ZZ) \to \UT_4(\ZZ/m\ZZ)$ and prolonged with the natural homomorphism $\UT_4(\ZZ/m\ZZ) \to \UT_4(\ZZ/2\ZZ)$, all-together inducing a homomorphism of multipliers $\MM(\UT_n(\ZZ/m\ZZ)) \to \MM(\UT_4(\ZZ/2\ZZ))$. The commutator $[[s_1, s_2], [s_2, s_3]]$ is thus trivial in the multiplier of $\UT_4(\ZZ/2\ZZ)$. By what we have shown so far, we should therefore have $\MM(\UT_4(\ZZ/2\ZZ)) \cong (\ZZ/2\ZZ)^3$, a contradiction.
\end{proof}

Lemmas \ref{lem:expansion} and \ref{lem:unique} combined prove our theorem.
 
\begin{theorem} \label{thm:main}
 The Schur multiplier of $\UT_n(\ZZ/m\ZZ)$ is isomorphic to
	\[
		\textstyle (\ZZ/m\ZZ)^{\binom{n}{2} - 1} \text{ for odd $m$,} \qquad (\ZZ/2\ZZ)^{n-3} \oplus (\ZZ/\frac{m}{2}\ZZ)^{n-2} \oplus (\ZZ/m\ZZ)^{\binom{n-1}{2}}  \text{ for even $m$.}
	\]
\end{theorem}

The methods above can also be applied to the group of unitriangular matrices over the ring $\ZZ$. Its presentation $F/R$ is just as that of the group $\UT_n(\ZZ/m\ZZ)$, only the relators $s_i^m = 1$ should be removed. Lemma \ref{lem:orders} is transcribed into the fact that the relators $[[s_i, s_{i+1}], [s_{i+1}, s_{i+2}]]$ are of order dividing $2$ in the multiplier, and Lemma \ref{lem:expansion} gives us a canonical expansion of a given element in $R \cap [F,F]$, where only the exponents of $[[s_i, s_{i+1}], [s_{i+1}, s_{i+2}]]$ are restricted. The uniqueness lemma goes as follows.

\begin{lemma} \label{lem:unique2}
Suppose that for integers $a_{ij}, b_i, c_i$ and nonnegative integers $d_i < 2$, the product
 \[
    \prod_{i < j-1} [s_i, s_j]^{a_{ij}} \cdot \prod_{i} [s_i, s_{i+1}, s_i]^{b_i} \cdot \prod_{i} [s_i, s_{i+1}, s_i s_{i+1}^{-1}]^{c_i} \cdot \prod_{i} [[s_i, s_{i+1}], [s_{i+1}, s_{i+2}]]^{d_i}
 \]
is contained in $[R,F]$. Then all $a_{ij}$, $b_i$, $c_i$ and $d_i$ are zero.
\end{lemma}
\begin{proof}
Pick any even positive integer $m$ greater than $\max \{ |a_{ij}|, |b_i|, |c_i|, 2 \}$. There is a canonical epimorphism $\UT_n( \ZZ) \to \UT_n(\ZZ/m\ZZ)$, inducing a homomorphism of multipliers. The product of the lemma gets mapped into the situation of Lemma \ref{lem:unique} and the proof is thus complete.
\end{proof}

\begin{theorem} \label{thm:main2}
The Schur multiplier of $\UT_n( \ZZ)$ is isomorphic to 
\[
	(\ZZ/2\ZZ)^{n-3} \oplus \ZZ^{\frac{(n+1)(n-2)}{2}}.	
\]
\end{theorem}

Alternatively, one can check that the preceding lemmas can be suitably adapted when taking $m = 0$, thereby directly incorporating Theorem \ref{thm:main2} into Theorem \ref{thm:main}.

The obtained results may be applied to the questions regarding minimality (in terms of the number of generators and relators) of the above presentations of groups of unitriangular matrices over $\ZZ/2\ZZ$ and $\ZZ$. Note that minimality over $\ZZ/2\ZZ$ implies minimality over $\ZZ$. In \cite{Biss01}, this is done for presentations of the groups $\UT_n(\ZZ/q\ZZ)$ with $q$ an odd prime by computing the $p$-rank of $\HH^2(\UT_n(\ZZ/q\ZZ), \FF_p)$ using a description of the Schur multiplier. Relying on Theorem \ref{thm:main}, one may check that the very same methods work in the even case $\ZZ/2\ZZ$.

\begin{corollary}
The given presentations of the groups $\UT_n(\ZZ/2\ZZ)$ and $\UT_n(\ZZ)$ are minimal.
\end{corollary}


%
%
%
%

\section{Multipliers over products of cyclic rings}
\numberwithin{equation}{section}
\noindent In this section, we extend Theorem \ref{thm:main} by determining the Schur multiplier of the group $\UT_n(\prod_{i=1}^r \ZZ/m_i\ZZ)$, where we assume $m_i | m_{i+1}$. As with Theorem \ref{thm:main2}, one may take $m_i = 0$ here and the result remains valid. The bulk of the calculation consists of the case $r = 2$, the rest is done inductively.

So assume first that $r = 2$, the group in question being $\UT_n(\ZZ/m_1\ZZ \times \ZZ/m_2\ZZ)$ with $m_1$ dividing $m_2$.  The presentation of this group is given similarly as in the previous section, c.f. \cite{Biss01}. The generators may be chosen to be $\mathcal S = \{ s_i(1) \mid 1 \leq i \leq n-1 \} \cup \{ s_i(x) \mid 1 \leq i \leq n-1 \}$, where $s_i(\lambda)$ is meant to represent the elementary matrix $I + \lambda E_{i,i+1}$ in the unitriangular group, and $1 \equiv (1,1), x \equiv (0,1)$. The set of relators $\mathcal R$ consists of elements in Table \ref{tab:relators2}, where the notation $s_i(1|x)$ denotes any element of the set $\{ s_i(1), s_i(x)\}$.
\begin{table}[h]
	\caption{Relators in the presentation of $\UT_n(\ZZ/m_1 \times \ZZ/m_2)$.}
	\label{tab:relators2}
	\begin{center}
		\begin{tabular}{clr}
\toprule
$(\mathcal R.1)$	& $s_i(1)^{m_1} s_i(x)^{-m_1}$
 		& $1 \leq i \leq n-1$ \\

		& $s_i(x)^{m_2}$
 		& $1 \leq i \leq n-1$ \\

$(\mathcal R.2)$	& $[s_i(1), s_i(x)]$ 
 		& $1 \leq i \leq n-1$ \\

$(\mathcal R.3)$	& $[s_i(1|x), s_j(1|x)]$
		& $1 \leq i < j-1 \leq n-2$ \\

$(\mathcal R.4)$	& $[s_i(1|x)^{-1}, s_{i+1}(x)^{-1}] [s_i(x), s_{i+1}(1)]^{-1}$
		& $1 \leq i \leq n-2$ \\

$(\mathcal R.5)$	& $[s_i(1|x), s_{i+1}(1), s_i(1)]$
		& $1 \leq i \leq n-2$ \\

		& $[s_i(1|x), s_{i+1}(1), s_{i+1}(1)]$
		& $1 \leq i \leq n-2$ \\

$(\mathcal R.6)$	& $[[s_i(1|x), s_{i+1}(1)], [s_{i+1}(1|x), s_{i+2}(1)]]$
		& $1 \leq i \leq n-3$  \\ \bottomrule
		\end{tabular}
	\end{center}
\end{table}

First off, we establish a bound on the orders of relators gathered in Table \ref{tab:2}.

\begin{table}[htb]
\caption{Generators of $\MM(\UT_n(\ZZ/m_1\ZZ \times \ZZ/m_2\ZZ))$ and their orders.}
\label{tab:2}
	\begin{center}
		\begin{tabular}{clr}
\toprule
$(1)$	& $[s_i(1), s_i(x)]$
	& $m_1$ \\ \midrule

$(2)$	& $[s_i(1), s_j(1)]$
	& $m_2$ \\

	& $[s_i(x), s_j(x)][s_i(1), s_j(1)]^{-1}$
	& $m_1$ \\

	& $[s_i(x), s_j(1)][s_i(x), s_j(x)]^{-1}$
	& $m_1$ \\

	& $[s_i(1), s_j(x)][s_i(x), s_j(x)]^{-1}$
	& $m_1$ \\ \midrule

$(3)$	& $[s_i(1), s_{i+1}(x)]  [s_i(x), s_{i+1}(1)]^{-1}$
	& $m_1$ \\ 

	& $[s_i(x), s_{i+1}(x)]  [s_i(x), s_{i+1}(1)]^{-1}$
	& $m_1$ \\ \midrule 

$(4)$	& $[s_i(1), s_{i+1}(1), s_i(1)]$
	& $m_2$ \\

	& $[s_i(1), s_{i+1}(1), s_{i+1}(1) s_i(1)^{-1}]$
	& $\gcd(m_2, \binom{m_2}{2})$ \\

	& $[s_i(1), s_{i+1}(1), s_i(x)s_i(1)^{-1}]$
	& $m_1$ \\

	& $[s_i(1), s_{i+1}(1), s_i(x)s_{i+1}(x)^{-1} (s_i(1)s_{i+1}(1)^{-1})^{-1}]$
	& $\gcd(m_1, \binom{m_1}{2})$  \\ \midrule

$(5)$	& $[s_i(1), s_{i+1}(1), s_{i+2}(1), s_{i+1}(1)]$
	& $\gcd(2,m_2)$\\

	& $[s_i(1), s_{i+1}(1), s_{i+2}(1), s_{i+1}(x)s_{i+1}(1)^{-1}]$
	& $\gcd(2,m_1)$ \\   \bottomrule
		\end{tabular}
	\end{center}
\end{table}

\begin{lemma} \label{lem:orders2}
 In the group $F/[R,F]$, elements of Table \ref{tab:2} have orders dividing the numbers in the far right column.
 \end{lemma}
\begin{proof}
All calculations are done in the group $F/[R,F]$. We first have $[s_i(1), s_i(x)]^{m_1} = [s_i(1)^{m_1}, s_i(x)] = 1$, which proves $(1)$. Next, the congruence $s_i(1)^{m_2} \equiv s_i(x)^{m_2} \equiv 1 \pmod{R}$ implies $[s_i(1), s_j(1)]^{m_2} = 1$. The other three orders in $(2)$ are immediate from this one. For $(3)$, note that the product $[s_i(1), s_{i+1}(x)] [s_i(x), s_{i+1}(1)]^{-1}$ is trivial in the group of unitriangular matrices, hence central in $F/[R,F]$. The commutator $[[s_i(1), s_{i+1}(x)], [s_i(x), s_{i+1}(1)]]$ is thus trivial, so it suffices to prove the equality $[s_i(1), s_{i+1}(x)]^{m_1} = [s_i(x), s_{i+1}(1)]^{m_1}$, and similarly $[s_i(x), s_{i+1}(x)]^{m_1} = [s_i(x), s_{i+1}(1)]^{m_1}$ for the other relator. Using relation $(\mathcal R.1)$, we have
\begin{align*}
&[s_i(1)^{m_1}, s_{i+1}(x)] \\
&\quad = [s_i(1), s_{i+1}(x)]^{m_1} \cdot [s_i(1), s_{i+1}(x), s_i(1)]^{\binom{m_1}{2}} \\
&\quad = [s_i(x), s_{i+1}(x)]^{m_1} \cdot [s_i(x), s_{i+1}(x), s_i(x)]^{\binom{m_1}{2}}.
\end{align*}
Note that the factors of weight three are equal by relation $(\mathcal R.2)$, since
\begin{align*}
&[s_i(1), s_{i+1}(x), s_i(1)] \\
&\quad = [s_i(x), s_{i+1}(x), s_i(1)] \\
&\quad = [s_i(x)^{-1} s_i(x)^{s_{i+1}(x)}, s_i(1)] \\
&\quad = [s_i(x)^{s_{i+1}(x)}, s_i(1)] \\
&\quad = [s_i(x), s_i(1) [s_i(1), s_{i+1}(x)^{-1}]] \\
&\quad = [s_i(1), s_{i+1}(x), s_i(x)] \\
&\quad = [s_i(x), s_{i+1}(x), s_i(x)].
\end{align*}
This implies the equality $[s_i(1), s_{i+1}(x)]^{m_1} = [s_i(x), s_{i+1}(x)]^{m_1}$. In a similar fashion, we have
\begin{align*}
&[s_i(x), s_{i+1}(x)^{m_1}] \\
&\quad = [s_i(x), s_{i+1}(x)]^{m_1} \cdot [s_i(x), s_{i+1}(x), s_{i+1}(x)]^{\binom{m_1}{2}} \\
&\quad = [s_i(x), s_{i+1}(1)]^{m_1} \cdot [s_i(x), s_{i+1}(1), s_{i+1}(1)]^{\binom{m_1}{2}}
\end{align*}
and
\begin{align*}
& [s_i(x), s_{i+1}(1), s_{i+1}(1)] & \\
&\quad = [s_i(x), s_{i+1}(x), s_{i+1}(1)] \\
&\quad = [s_{i+1}(x)^{-s_i(x)} s_{i+1}(x), s_{i+1}(1)] \\
&\quad = [s_{i+1}^{-s_i(x)}, s_{i+1}(1)] \\
&\quad = [s_{i+1}(x)^{-1}, s_{i+1}(1)[s_{i+1}(1), s_i(x)^{-1}]] \\
&\quad = [s_{i+1}(x), [s_{i+1}(1), s_i(x)^{-1}]]^{-1} \\
&\quad = [s_i(x), s_{i+1}(1), s_{i+1}(x)] \\
&\quad = [s_i(x), s_{i+1}(x), s_{i+1}(x)],
\end{align*}
which implies  $[s_i(x), s_{i+1}(x)]^{m_1} = [s_i(x), s_{i+1}(1)]^{m_1}$. This concludes the proof of $(3)$. The first two relators of $(4)$ are checked as in the proof of Lemma \ref{lem:orders}, and the third one is straightforward. For the last one, expand the trivial commutator $[(s_i(1) s_i(x)^{-1})^{m_1}, s_{i+1}(1) s_{i+1}(x)^{-1}]$ into
\[
[s_i(1) s_i(x)^{-1}, s_{i+1}(1) s_{i+1}(x)^{-1}]^{m_1} \cdot [s_i(1) s_i(x)^{-1}, s_{i+1}(1) s_{i+1}(x)^{-1}, s_i(1) s_i(x)^{-1}]^{\binom{m_1}{2}}
\]
and the commutator $[s_i(1) s_i(x)^{-1}, (s_{i+1}(1) s_{i+1}(x)^{-1})^{m_1}]$ into
\[
[s_i(1) s_i(x)^{-1}, s_{i+1}(1) s_{i+1}(x)^{-1}]^{m_1} \cdot [s_i(1) s_i(x)^{-1}, s_{i+1}(1) s_{i+1}(x)^{-1}, s_{i+1}(1) s_{i+1}(x)^{-1}]^{\binom{m_1}{2}}.
\]
Now note that both factors of weight three can be simplified. We have
\begin{align*}
& [s_i(1) s_i(x)^{-1}, s_{i+1}(1) s_{i+1}(x)^{-1}, s_i(1) s_i(x)^{-1}] \\
&\quad =  [[s_i(1), s_{i+1}(1) s_{i+1}(x)^{-1}]^{s_i(x)^{-1}} [s_i(x)^{-1}, s_{i+1}(1) s_{i+1}(x)^{-1}], s_i(1) s_i(x)^{-1}] \\
&\quad = [s_i(1), s_{i+1}(1) s_{i+1}(x)^{-1}, s_i(1) s_i(x)^{-1}] [s_i(x), s_{i+1}(1)s_{i+1}(x)^{-1},s_i(1)s_i(x)^{-1}]^{-1},
\end{align*}
where the second factor is trivial, since
\begin{align*}
&  [s_i(x), s_{i+1}(1)s_{i+1}(x)^{-1},s_i(1)s_i(x)^{-1}] \\
&\quad = [[s_i(x), s_{i+1}(x)^{-1}] [s_i(x), s_{i+1}(1)]^{s_{i+1}^{-1}}, s_i(1) s_i(x)^{-1}] \\
&\quad = [s_i(x), s_{i+1}(x)^{-1}, s_i(1) s_i(x)^{-1}]^* [[s_i(x), s_{i+1}(1)]^{s_{i+1}(1)^{-1}}, s_i(1) s_i(x)^{-1}] \\
&\quad = [s_i(x), s_{i+1}(x), s_i(1) s_i(x)^{-1}]^{-*} [s_i(x), s_{i+1}(1), s_i(1) s_i(x)^{-1}],
\end{align*}
and both of these are trivial by calculations made in part $(3)$. Hence
\begin{align*}
& [s_i(1) s_i(x)^{-1}, s_{i+1}(1) s_{i+1}(x)^{-1}, s_i(1) s_i(x)^{-1}] \\
&\quad = [s_i(1), s_{i+1}(1) s_{i+1}(x)^{-1}, s_i(1) s_i(x)^{-1}] \\
&\quad = [[s_i(1), s_{i+1}(x)^{-1}][s_i(1), s_{i+1}(1)]^{s_{i+1}(1)^{-1}}, s_i(1) s_i(x)^{-1}] \\
&\quad = [s_i(1), s_{i+1}(x)^{-1}, s_i(1) s_i(x)^{-1}]^* [s_i(1), s_{i+1}(1), s_i(1) s_i(x)^{-1}] \\
&\quad = [s_i(1), s_{i+1}(1), s_i(1) s_i(x)^{-1}]
\end{align*}
and analogously
\begin{align*}
& [s_i(1) s_i(x)^{-1}, s_{i+1}(1) s_{i+1}(x)^{-1}, s_{i+1}(1) s_{i+1}(x)^{-1}] \\
&\quad = [s_i(1), s_{i+1}(1), s_{i+1}(1) s_{i+1}(x)^{-1}].
\end{align*}
Thus the order of the relator $[s_i(1), s_{i+1}(1), s_i(x)s_{i+1}(x)^{-1} (s_i(1)s_{i+1}(1)^{-1})^{-1}]$ indeed divides both $m_1$ and $\binom{m_1}{2}$. Finally, we apply calculations made in the proof of Lemma \ref{lem:orders} to obtain $(5)$. Note that
\begin{align*}
  [s_{i}(1|x), s_{i+1}(1), s_{i+2}(1), s_{i+1}(1|x)] =  [s_{i+2}(1), s_{i+1}(1|x), s_{i}(1|x), s_{i+1}(1)]^{-1}
\end{align*}
via the commutator $[[s_i(1|x), s_{i+1}(1)], [s_{i+1}(1|x), s_{i+2}(1)]]$. This is in turn equal to $[s_{i}(1), s_{i+1}(1), s_{i+2}(1), s_{i+1}(x)]$ if $x$ appears at least once, as the left commutant only depends on representatives modulo $R$, and to $[s_{i}(1), s_{i+1}(1), s_{i+2}(1), s_{i+1}(1)]$  otherwise. These two are of order dividing $2$ by the same arguments as those made in the cyclic case.
\end{proof}

Next, we prove these commutators are in fact a generating set.

\begin{lemma} \label{lem:expansionProducts}
Relators of Table \ref{tab:2} generate the group $(R\cap[F,F])/[R,F]$.
\end{lemma}
\begin{proof}
 Let $w \in R \cap [F,F]$. Expand it as a product of conjugates of elements of $\mathcal R$. As $w$ is contained in $[F,F]$, the relators $(\mathcal R.1)$ cancel out. The next two, $(\mathcal R.2)$ and $(\mathcal R.3)$, are clear. Relators $(\mathcal R.4)$ can be obtained from Table \ref{tab:2} as
\begin{align*}
& [s_i(1)^{-1}, s_{i+1}(x)^{-1}] [s_i(x), s_{i+1}(1)]^{-1} \\
&\quad = [s_i(1), s_{i+1}(x)] [s_i(1), s_{i+1}(x), s_{i+1}(x)^{-1} s_i(1)^{-1}] [s_i(x), s_{i+1}(1)]^{-1} \\
&\quad = [s_i(1), s_{i+1}(x)] [s_i(x), s_{i+1}(1)]^{-1} [s_i(1), s_{i+1}(x), s_i(1)]^{-1} [s_i(1), s_{i+1}(x), s_{i+1}(x)]^{-1} \\
&\quad = [s_i(1), s_{i+1}(x)] [s_i(x), s_{i+1}(1)]^{-1} [s_i(1), s_{i+1}(1), s_i(x)]^{-1} [s_i(1), s_{i+1}(1), s_{i+1}(x)]^{-1},
\end{align*}
and relators $(\mathcal R.5)$ and $(\mathcal R.6)$ from the calculations made in lemmas \ref{lem:orders2} and \ref{lem:orders}.
\end{proof}

Finally, the uniqueness lemma asserts that relators of Table \ref{tab:2} are independent in the multiplier.

\begin{lemma} \label{lem:uniqueProducts}
Whenever a product of relators of Table \ref{tab:2} is an element of $[R,F]$, the corresponding exponents are divisible by the numbers in the far right column of Table \ref{tab:2}.
\end{lemma}
\begin{proof}
Let $w$ be such a combination of relators of Table \ref{tab:2}. Note that we may assume the exponents of these relators are bounded by the numbers in the right column of the table by Lemma \ref{lem:orders2}. Consider the homomorphism $\UT_n( \ZZ/m_1\ZZ \times \ZZ/m_2\ZZ) \to \UT_n( \ZZ/m_2\ZZ)$, induced by the projection homomorphism of the corresponding rings. The generators $s_i(1|x)$ get mapped into generators $s_i$, and the element $w$, being trivial in the multiplier, gets mapped into a trivial element of the multiplier of $\UT_n( \ZZ/m_2\ZZ)$. By Lemma \ref{lem:unique}, the exponents of the relators in $w$ that do not involve $x$ must thus be zero. Now consider the homomorphism $\UT_n( \ZZ/m_1\ZZ \times \ZZ/m_2\ZZ) \to \UT_n( \ZZ/m_1\ZZ)$. The generators $s_i(x)$ are in the kernel and $s_i(1)$ get mapped into generators $s_i$. Again, the element $w$ is trivial in the multiplier of $\UT_n( \ZZ/m_1\ZZ)$, and what remains of the generating relators of Table \ref{tab:2} are commutators $[s_i, s_j]^{-1}$ (coming from part $(2)$ of Table \ref{tab:2}), $[s_i, s_{i+1}, s_i^{-1}]$ and $[s_i, s_{i+1}, s_{i+1}s_i^{-1}]$ (from $(4)$) and $[s_i, s_{i+1}, s_{i+2}, s_{i+1}^{-1}]$ (from $(5)$). Therefore the exponents of these relators in $w$ must all be zero. Only commutators of weight $2$ and orders dividing $m_1$ in Table \ref{tab:2} may thus appear in $w$. As $w$ is also contained in $[R,F]$, it can be written as a product of conjugates of commutators, the first commutant of which is a relator. Modulo $\gamma_3(F)$, this amounts to
\[
\prod_{i,j} [s_i(1), s_j(1|x)]^{m_2} \cdot \prod_{i,j} ([s_i(1), s_j(1)][s_i(x), s_j(1)]^{-1})^{m_1} \cdot \prod_{i,j} ([s_i(1), s_j(x)][s_i(x), s_j(x)]^{-1})^{m_1}.
\]
Any such product contains only commutators with exponents divisible by $m_1$. Since basic commutators form a basis of the free abelian group $\gamma_2(F)/\gamma_3(F)$, the rest of the exponents must all be zero by the restriction on the exponents appearing in $w$. The proof is complete.
\end{proof}

Taking Theorem \ref{thm:main} in mind, there are $(n-1) + 2\binom{n-2}{2} + 2(n-2) = (n-1)^2$ additional relators of order $m_1$ in Table \ref{tab:2} besides those that make up the multipliers of $\UT_n(\ZZ/m_1\ZZ)$ and $\UT_n(\ZZ/m_2\ZZ)$. We have thus proved the following. 

\begin{theorem} \label{thm:main3}
The Schur multiplier of $\UT_n(\ZZ/m_1\ZZ \times \ZZ/m_2\ZZ)$, where $m_1 | m_2$, is isomorphic to
\[
\MM(\UT_n( \ZZ/m_1\ZZ)) \oplus \MM(\UT_n( \ZZ/m_2\ZZ)) \oplus (\ZZ/m_1\ZZ)^{(n-1)^2}.
\]
\end{theorem}


Again, the methods can be applied to the group of unitriangular matrices over the rings $\ZZ/m\ZZ \times \ZZ$ or $\ZZ \times \ZZ$. The presentation of such a group is the same as the one for $\UT_n(\ZZ/m_1\ZZ \times \ZZ/m_2\ZZ)$, only relators $s_i(x)^{m_2}=1$ and respectively $s_i(1)^{m_1} = s_i(x)^{m_1}$ should be removed. The three lemmas are just as those leading up to Theorem \ref{thm:main2}, so Theorem \ref{thm:main3} remains valid for $m_1=m_2=0$ or $m_2 = 0$. This concludes our calculation in the case $r=2$. 

We now demonstrate the transition from $r = 2$ to $r = 3$, as the general case is done identically, only more indices are required. The group of unitriangular matrices $\UT_n(\ZZ/m_1\ZZ \times \ZZ/m_2\ZZ \times \ZZ/m_3\ZZ)$ may be generated by the set $\mathcal S$ of generators of the group $\UT_n(\ZZ/m_1\ZZ \times \ZZ/m_2\ZZ)$ together with a set of additional generators $\{ s_i(y) \mid 1 \leq i \leq n-1 \}$. The elements $1,x,y$ respectively represent $(1,1,1), (0,1,1), (0,0,1)$ in the ring $\ZZ/m_1\ZZ \times \ZZ/m_2\ZZ \times \ZZ/m_3\ZZ$. The essential relators $(\mathcal R.1)$ become $s_i(1)^{m_1} = s_i(x)^{m_1}$, $s_i(x)^{m_2} = s_i(y)^{m_2}$ and $s_i(y)^{m_3} = 1$. Using epimorphisms of multipliers deriving from the projections just as in the proof of  Lemma \ref{lem:uniqueProducts}, we conclude that the corresponding basis of the multiplier is given by Table \ref{tab:3}. Note that relators $(3)$ of the table guarantee that commutators with at least one appearance of $s_i(y)$ are equal to those where only the $s_i(y)$ appear, and these are in turn equal to those commutators in which the generators with $y$ appear only in the last place and the rest are all equal to $s_i(1)$.

\begin{table}[htb]
	\caption{Generators of $\MM(\UT_n(\ZZ/m_1\ZZ \times \ZZ/m_2\ZZ \times \ZZ/m_3\ZZ))$ and their orders.}
	\label{tab:3}
	\begin{center}
		\begin{tabular}{clr}
\toprule
$(1)$	& $[s_i(1), s_i(x)]$				
	& $m_1$ \\
	& $[s_i(1), s_i(y)]$				
	& $m_2$  \\
	& $[s_i(x), s_i(y)]$					
	& $m_2$ \\ \midrule

$(2)$ & $[s_i(1), s_j(1)]				$
	& $m_3$ \\
	& $[s_i(x), s_j(x)][s_i(1), s_j(1)]^{-1}		$
	& $m_1$ \\ 
	& $[s_i(y), s_j(y)][s_i(1), s_j(1)]^{-1}		$
	& $m_2$ \\
	& $[s_i(x), s_j(1)][s_i(x), s_j(x)]^{-1}		$
	& $m_1$ \\
	& $[s_i(1), s_j(x)][s_i(x), s_j(x)]^{-1}		$
	& $m_1$ \\  
	& $[s_i(y), s_j(1)][s_i(y), s_j(y)]^{-1}		$
	& $m_2$ \\
	& $[s_i(1), s_j(y)][s_i(y), s_j(y)]^{-1}		$
	& $m_2$ \\ 
	& $[s_i(y), s_j(x)][s_i(y), s_j(y)]^{-1}		$
	& $m_2$ \\
	& $[s_i(x), s_j(y)][s_i(y), s_j(y)]^{-1}		$
	& $m_2$ \\  \midrule

$(3)$	& $[s_i(1), s_{i+1}(x)]  [s_i(x), s_{i+1}(1)]^{-1}			$
	& $m_1$ \\ 
	& $[s_i(x), s_{i+1}(x)]  [s_i(x), s_{i+1}(1)]^{-1}			$
	& $m_1$ \\
	& $[s_i(y), s_{i+1}(x)]  [s_i(y), s_{i+1}(1)]^{-1}			$
	& $m_2$ \\ 
	& $[s_i(1), s_{i+1}(y)]  [s_i(y), s_{i+1}(1)]^{-1}			$
	& $m_2$ \\
	& $[s_i(x), s_{i+1}(y)]  [s_i(y), s_{i+1}(1)]^{-1}			$
	& $m_2$ \\ 
	& $[s_i(y), s_{i+1}(y)]  [s_i(y), s_{i+1}(1)]^{-1}			$
	& $m_2$ \\ \midrule

$(4)$ & $[s_i(1), s_{i+1}(1), s_i(1)]			$
	& $m_3$ \\
	& $[s_i(1), s_{i+1}(1), s_{i+1}(1)s_i(1)^{-1}]$
	& $\gcd(m_3, \binom{m_3}{2})$ \\
	& $[s_i(1), s_{i+1}(1), s_i(1)s_i(x)^{-1}]	$
	& $m_1$ \\
	& $[s_i(1), s_{i+1}(1), s_i(1)s_{i+1}(1)^{-1} (s_i(x)s_{i+1}(x)^{-1})^{-1}]$ 
	& $\gcd(m_1, \binom{m_1}{2})$  \\
	& $[s_i(1), s_{i+1}(1), s_i(1)s_i(y)^{-1}]	$
	& $m_2$ \\
	& $[s_i(1), s_{i+1}(1), s_i(1)s_{i+1}(1)^{-1} (s_i(y)s_{i+1}(y)^{-1})^{-1}]$ 
	& $\gcd(m_2, \binom{m_2}{2})$  \\ \midrule

$(5)$ & $[s_i(1), s_{i+1}(1), s_{i+2}(1), s_{i+1}(1)]$
	& $\gcd(2,m_3)$ \\
	& $[s_i(1), s_{i+1}(1), s_{i+2}(1), s_{i+1}(1)s_{i+1}(x)^{-1}]$
	& $\gcd(2,m_1)$ \\
	& $[s_i(1), s_{i+1}(1), s_{i+2}(1), s_{i+1}(1)s_{i+1}(y)^{-1}]$
	& $\gcd(2,m_2)$ \\ \bottomrule
		\end{tabular}		
	\end{center}
\end{table}

In the general case, the unitriangular group $\UT_n(\prod_{i=1}^r \ZZ/m_i\ZZ)$ may be generated by the set $\mathcal S = \{ s_i(x_k) \mid 1 \leq i \leq n-1, 1 \leq k \leq r \}$, where the element $x_k \in  \prod_{i=1}^r \ZZ/m_i\ZZ$ represents the $r$-tuple consisting of $0$ in the first $k-1$ places and of $1$ in the rest. Inductively, the corresponding table that determines the basis of the multiplier is partitioned into five parts just as Table \ref{tab:3} is, and on each step, additional relators are added. These are collected in Table \ref{tab:add}.

\begin{table}[htb]
\caption{Additional generators of $\MM(\UT_n(\prod_{i=1}^r \ZZ/m_i\ZZ))$.}
\label{tab:add}	
	\begin{center}
		\begin{tabular}{clr}
\toprule
$(1)$	& $[s_i(x_k), s_i(x_r)]$ \\ 

$(2)$ & $[s_i(x_r), s_j(x_r)][s_i(x_1), s_j(x_1)]^{-1}$ \\ 

	& $[s_i(x_{k_1}), s_j(x_{k_2})][s_i(x_r),s_j(x_r)]^{-1}$
	& $r \in \{ k_1, k_2 \}$ \\ 

$(3)$	& $[s_i(x_{k_1}), s_{i+1}(x_{k_2})][s_i(x_r), s_{i+1}(x_1)]$
	& $k_2 \neq 1$, $r \in \{ k_1, k_2 \}$ \\ 

$(4)$ & $[s_i(x_1), s_{i+1}(x_1), s_i(x_1)s_i(x_r)^{-1}]$ \\

	& $[s_i(x_1), s_{i+1}(x_1), s_i(x_1)s_{i+1}(x_1)^{-1} (s_i(x_r)s_{i+1}(x_r)^{-1})^{-1}]$ \\ 

$(5)$ & $[s_i(x_1), s_{i+1}(x_1), s_{i+2}(x_1), s_{i+1}(x_1)s_{i+1}(x_r)^{-1}]$\\ \bottomrule
		\end{tabular}		
	\end{center}
\end{table}

All in all, the number of these additional generators on step $r$ not counting the ones that make up the multiplier of the unitriangular group over $\prod_{i=1}^{r-1} \ZZ/m_i\ZZ$ and over $\ZZ/m_r\ZZ$ equals $(r-1)(n-1) + 2(r-1)\binom{n-2}{2} + 2(r-1)(n-2) = (r-1)(n-1)^2$. They are all of order $m_{r-1}$. Inductively, we have proven the following.

\begin{theorem}
\label{t:main_all}
The Schur multiplier of $\UT_n(\prod_{i=1}^r \ZZ/m_i\ZZ)$, where $m_i | m_{i+1}$, is isomorphic to
 	\[
		\bigoplus_{i=1}^r \MM(\UT_n( \ZZ/m_i\ZZ)) \oplus \bigoplus_{i=1}^{r-1} (\ZZ/{m_i}\ZZ)^{i(n-1)^2}.
	\]
\end{theorem}

Using the above, we can also determine the Schur multiplier of the ``limit'' unitriangular groups with respect to $n$. Somewhat in the spirit of algebraic $K$-theory, the {\em infinite} unitriangular group $\UT(R)$ of a unital ring $R$ may be defined as the direct limit of the groups $\UT_n( R)$ with upper-left embeddings $\UT_n(R) \to \UT_{n+1}(R)$. Applying the Direct Limit Argument \cite{Beyl82}, the following is straightforward from Theorem \ref{t:main_all}.

\begin{corollary} \label{cor:directlimit}
The Schur multiplier of $\UT(\ZZ/m\ZZ)$ is isomorphic to 
\[
\bigoplus_{i \in \mathbb{N}} \ZZ/m\ZZ \text{ for odd $m$,} \qquad \bigoplus_{i \in \mathbb{N}} \ZZ/2\ZZ \oplus \ZZ/{\textstyle \frac{m}{2}}\ZZ \oplus \ZZ/m\ZZ \text{ for even $m$.}
\]
The Schur multiplier of the infinite unitriangular group over direct products of such rings is isomorphic to the product of multipliers of the infinite unitriangular group over the corresponding factors.
\end{corollary}
\section{Covering groups over prime fields}
\label{s:covers}
\noindent In this section, we determine the covering groups of the group of unitriangular matrices over the field $\FF_p$ with the help of Theorem \ref{thm:main}. Recall that a {\em covering group} of a given group $G$ is a stem extension (i.e. the kernel is contained in the center and the derived subgroup) with the quotient group $G$ that is of maximal order among such extensions \cite{Kar87}. It is well-known that all covering groups of $G$, presented as $F/R$, are homomorphic images of the group $F/[R,F]$. Moreover, any one of them is obtained as a quotient $F/C$, where $C$ is the torsion-free part of $R/[R,F]$, or equivalently the complement of the multiplier $\MM(G) \cong (R \cap [F,F])/[R,F]$ in $R/[R,F]$.

Invoking the minimal presentations of the previous section, let $\UT_n(\FF_p)$ be given by the set of generators $\mathcal S = \{ s_i \mid 1 \leq i \leq n-1 \}$ and the set $\mathcal R$ of relators $s_i^m$, $[s_i, s_j]$, $[s_i, s_{i+1}, s_i]$, $[s_i, s_{i+1}, s_{i+1}]$, and $[[s_i, s_{i+1}], [s_{i+1}, s_{i+2}]]$ for suitable indices $i,j$. Taking Theorem \ref{thm:main} into account, we denote $\alpha_{i,j} = [s_i, s_j]$, $\beta_{i,0} = [s_i, s_{i+1}, s_i]$, $\beta_{i,1} = [s_i, s_{i+1}, s_is_{i+1}^{-1}]$, and $\gamma_i = [[s_i, s_{i+1}], [s_{i+1}, s_{i+2}]]$. The group $R/[R,F]$ is generated by these commutators and powers $s_i^p$, and the multiplier corresponds precisely to the subgroup generated by commutators alone. The subgroup $\langle s_i^p \mid 1 \leq i \leq n-1 \rangle \leq R/[R,F]$ is thus a free complement of the multiplier. Moreover, any free complement can be, by means of Gaussian elimination, generated by a set of the form
\begin{equation}
\label{eq:relators_covers1}
	\{ s_i^p \cdot \alpha_{i,j}^{a_{i,j}} \beta_{i,j}^{b_{i,j}} \gamma_i^{c_i} \mid 1 \leq i \leq n-1 \}
\end{equation}
for some nonnegative integers $a_{i,j}, b_{i,0} < p$, $b_{i,1} < p$ for odd $p$ and $0$ otherwise, $c_i < 2$ for even $p$ and $0$ otherwise. Any such complement provides a covering group of $\UT_n(\FF_p)$, given by the set of generators $\mathcal S$ and relators 
\begin{equation}
\label{eq:relators_covers}
	\{ [\alpha_{i,j}, \mathcal S] \} \cup \{ [\beta_{i,j}, \mathcal S] \} \cup \{ [\gamma_{i}, \mathcal S] \} \cup \{ [s_i^p, \mathcal S] \} \cup \eqref{eq:relators_covers1},
\end{equation}
and all covering groups of $\UT_n(\FF_p)$ are obtained this way. Note that relations $[s_i^p, \mathcal S]$ are actually consequences of the other four types of relations.

Any two covering group of a given group are known to be isoclinic \cite{Beyl82}. {\em Isoclinism} is an equivalence relation introduced by P. Hall, and the equivalence classes are called {\em families}. Each family contains {\em stem groups}, that is, groups $G$ satisfying $Z(G)\le [G,G]$. Stem groups in a given family have the same order, which is the minimal order of all groups in the family. We show by means of determining the center that covering groups of unitriangular groups possess this property.

\begin{proposition}
\label{p:center_covers}
The center of any covering group of $\UT_n(\FF_p)$ is generated by $\MM(\UT_n(\FF_p))$ when $n = 3$, or $4$ for $p=2$, and by $\MM(\UT_n(\FF_p)) \cup \{ [s_1, s_2, \dots, s_{n-1}] \}$ whenever $n \geq 4$, resp. $5$.
\end{proposition}
\begin{proof}
Let $H$ be a covering group of $\UT_n(\FF_p)$ and $w \in Z(H)$. Regard $w$ as an element $\tilde w \in F$, a word in the generators $\mathcal S$. We have $[\tilde w, \mathcal S] \subseteq R$, so $\tilde w$ is a central element of $\UT_n(\FF_p)$. The center of the unitriangular group is generated by the right-upper-most commutator $\lambda = [s_1, s_2, \dots, s_{n-1}]$. The element $\tilde w$, and so $w$, may thus be written as a product $\lambda^i r$, where $r \in R$. As $R$ is generated modulo $[R,F]$ by the set $\{ s_i^p, \alpha_{i,j}, \beta_{i,j}, \gamma_i \}$ and elements of the complement of the multiplier are relations in $H$, we conclude that $R$ is contained in the derived subgroup of $H$ and generated by the basis of the multiplier of $\UT_n(\FF_p)$. We now determine when such a product $w = \lambda^i r$ is central in $H$. This clearly reduces to the question of when $\lambda$ is central in $H$. When $n = 3$, the commutator $[s_1, s_2]$ is certainly not central, since we have $[s_1, s_2, s_1] = \beta_{1,0} \neq 1$. When $n = 4$, we have $[s_1, s_2, s_3, s_1] = [s_1, s_2, s_3, s_3] = 1$ by the Hall-Witt identity, but the commutator $[s_1, s_2, s_3, s_2]$ need not be trivial by Theorem \ref{thm:main}. If $p$ is odd, it is and $\lambda$ is central in $H$; if $p = 2$, this shows that $\lambda$ is not central in $H$. Now let $n \geq 5$. Again invoking the Hall-Witt identity, we see that $[\lambda, s_i] = 1$ whenever $1 \leq i \leq n-3$, and since $\lambda \equiv [s_{n-1}, s_{n-2}, \dots, s_1]^{\pm 1} \pmod{R}$, we also have $[\lambda, s_{n-2}] = [\lambda, s_{n-1}] = 1$, hence $\lambda$ is indeed central. 
\end{proof}

\begin{corollary}
\label{c:stem_covers}
Any covering group of $\UT_n(\FF_p)$ is a stem group.
\end{corollary}

We now consider the converse of Corollary \ref{c:stem_covers}. Let $H$ be a covering group of $\UT_n(\FF_p)$, given by the set of generators $\mathcal S$ and relators \eqref{eq:relators_covers} as above, and let $K$ be any group in the isoclinism family $\Phi$ determined by $H$. We ask whether or not $K$ must also be a covering group of the unitriangular group. Certainly, $K$ must be of the same order as $H$ is for the question to make sense, which means that $K$ must also be a stem group of the family $\Phi$. As the groups $H$ and $K$ have isomorphic derived subgroups, their centers are also isomorphic. Moreover, the isomorphism $H/Z(H) \cong K/Z(K)$ provides a generating set $K = \langle s_i, Z(K) \rangle$ with the relation $s_i^p \in Z(K)$. By Proposition \ref{p:center_covers}, this means that the elements
\begin{equation}
\label{eq:relators_isocline}
	s_i^p \cdot \lambda^{\tilde l_i} \alpha_{i,j}^{\tilde a_{i,j}} \tilde \beta_{i,j}^{b_{i,j}} \tilde \gamma_i^{c_i}
\end{equation}
are trivial in $K$ for some nonnegative integers $\tilde l_i, \tilde a_{i,j}, \tilde b_{i,0} < p$, $\tilde b_{i,1} < p$ for odd $p$ and $0$ otherwise, $\tilde c_i < 2$ for even $p$ and $0$ otherwise. The group $K$ is thus given by the set of generators $\mathcal S$ subject to the set of relations 
\[
	\{ [\alpha_{i,j}, \mathcal S] \} \cup \{ [\beta_{i,j}, \mathcal S] \} \cup \{ [\gamma_{i}, \mathcal S] \} \cup \{ [s_i^p, \mathcal S] \} \cup \eqref{eq:relators_isocline}.
\]
These differ from \eqref{eq:relators_covers} by the extra element $\lambda$. 

When $\tilde l_i = 0$ for all $i$, $K$ is indeed a covering group of $\UT_n(\FF_p)$. In light of Proposition \ref{p:center_covers}, this is always true for $n = 3$, so any stem group of the given isoclinism family is a covering group of $\UT_n(\FF_p)$. Moreover, as these covering groups are of small enough order, namely $|\MM(\UT_3(\FF_p))| \cdot |\UT_3(\FF_p)| = p^5$ by Theorem \ref{thm:main}, it is not difficult to classify the nonisomorphic representatives among them. This has been done in \cite{Jam80}. The covering groups of $\UT_3(\FF_p)$ correspond precisely to the stem members of the family $\Phi_6$.

\begin{proposition}
\label{p:number}
	The covering groups of $\UT_3(\FF_p)$ are precisely the stem groups of some isoclinism family. The number of nonisomorphic representatives of them equals $3$ for $p=2$, $7$ for $p=3$, and $p+7$ whenever $p > 3$.
\end{proposition}

We note here that the isoclinism family $\Phi_6$ plays a role in coclass theory. Its members appear among the top vertices of the coclass graph $\mathcal G(p,2)$ \cite{May12}, and thus play an important role as starting groups of the $p$-group generation algorithm \cite{OBrien90}.

When $n > 3$, determining nonisomorphic representatives of the covering groups in question becomes a lot more complicated and we were unable to provide analogous results in this case. We have tried approaching the problem by using the theory of central extensions of polycyclic groups \cite{Nickel93}. The group $\UT_n(\FF_p)$ may be presented by the set of generators $\{ {}^r\alpha_i \mid 1 \leq r \leq n-1, 1 \leq i \leq n-r \}$ subject to the power relations ${}^r\alpha_i^m = 1$ and commutator relations $[{}^r\alpha_i, {}^1\alpha_{i-1}] = {}^{r+1}\alpha_{i-1}^{-1}$ for $i \neq 1$, $[{}^r\alpha_i, {}^1\alpha_{i+r}] = {}^{r+1}\alpha_{i}^{-1}$ for $i+r \neq n$, and $[{}^r\alpha_i, {}^1\alpha_k] = 1$ otherwise.
%
%
Here, the generator ${}^r\alpha_i$ corresponds to the commutator $[s_i, s_{i+1}, \dots, s_{i+r-1}]$. The problem now essentially reduces to finding orbits of an action of a certain matrix group, see \cite[Theorem 7.4.8]{Nickel93}. Alas, these calculations turn out to be rather lengthy. In the smallest case of $\UT_4(\FF_3)$, we used {\sf GAP} \cite{GAP} to determine the number of nonisomorphic representatives of its covering groups; there are $47278$ of them. This seems to be quite in contrast with the growth of the number given in Proposition \ref{p:number} and indicates a more chaotic behaviour.
\vspace{0.5\baselineskip}

\noindent {\bf Acknowledgement.} ~ The author would like to thank Primož Moravec for his vital encouragement.

\end{document}